\documentclass[amssymb,reqno,amsfonts,refcheck,12pt,verbatim,righttag]{amsart}

\usepackage{graphicx}
\usepackage{graphics,color}
\usepackage{amssymb,mathrsfs}
\usepackage{graphicx,color,tikz,caption,subcaption}
\usepackage[all]{xy}
\usepackage{mathtools}
\usepackage{enumerate}
\usepackage{amsmath}
\usepackage{bbm}
\usepackage{mathrsfs}

\usepackage[colorlinks, linkcolor=blue,citecolor=blue, anchorcolor=blue,urlcolor=blue,pagebackref,hypertexnames=false]{hyperref}

\numberwithin{equation}{section} 
\newtheorem{thm}{Theorem}[section]
\newtheorem{lem}[thm]{Lemma}
\newtheorem{pro}[thm]{Proposition}

\newcommand{\R}{\mathbb{R}}

\newcommand{\N}{\mathbb{N}}
\newcommand{\Z}{\mathbb{Z}}

\usepackage{float}

\begin{document}
\baselineskip 14pt
\title{Multiplicative Diophantine approximation with  restricted denominators}

\author{Bing Li} 
\author{Ruofan Li} 
\author{Yufeng Wu*}
\thanks{*Corresponding author}

\address[]{School of Mathematics, South China University of Technology, Guangzhou,
510641, P. R. China}
\email{scbingli@scut.edu.cn}

\address[]{Department of Mathematics, Jinan University, Guangzhou, 510632, P. R. China}
\email{liruofan@jnu.edu.cn}

\address[]{School of Mathematics and Statistics\\HNP-LAMA\\Central South University\\ Changsha, 410083, P. R. China}
\email{yufengwu.wu@csu.edu.cn}


\thanks{Declarations of interest: none}

\date{}

\begin{abstract}
Let  $\{a_n\}_{n\in\N}$, $\{b_n\}_{n\in \N}$ be two infinite subsets of positive integers and $\psi:\N\to \R_{>0}$ be a positive function. We completely determine the Hausdorff dimensions of the set of all points $(x,y)\in [0,1]^2$ which satisfy $\|a_nx\|\|b_ny\|<\psi(n)$ infinitely often, and the set of all $x\in [0,1]$ satisfying $\|a_nx\|\|b_nx\|<\psi(n)$ infinitely often. This is based on establishing general convergence results for Hausdorff measures of these two sets. 
We also obtain some results on the set of all $x\in [0,1]$ such that $\max\{\|a_nx\|, \|b_nx\|\}<\psi(n)$ infinitely often.
\end{abstract}

\maketitle

\section{Introduction}\label{S-1}

In this paper, we mainly investigate  Hausdorff  dimensions of sets which arise in multiplicative Diophantine approximation. Given a nonnegative function $\psi:\N\to \R_{\geq0}$, a point $(x_1,\dots,x_d) \in \R^d$ is called {\em multiplicatively $\psi$-well approximable} if 
\begin{equation}\label{vb2}
\|nx_1\|\cdots\|nx_d\|<\psi(n)\quad \text{ for i.m. }n\in\N.
\end{equation}
Here and throughout, ``i.m.''   means ``infinitely many'', and  $ \| \cdot  \| $ denotes the distance of a real
number to the nearest integer. Denote by  $W^{\times}_d(\psi)$ the set of multiplicatively $\psi$-well approximable points in $[0,1]^d$. Notice that it causes no loss of generality to restrict to $[0,1]^d$, since  the set of multiplicatively $\psi$-well approximable points  is invariant under translations by integer vectors. 

There have been  many works on the metric and Hausdorff theory of $W^{\times}_d(\psi)$. Lots of works are motivated by a famous conjecture of  Littlewood in 1930s, which asserts that for any $\alpha,\beta\in\R$, one has
\[\liminf_{n\to\infty}n\|n\alpha\|\|n\beta\|=0.\]
The metric theory of $W^{\times}_d(\psi)$ was established by Gallagher \cite{Gallagher62}, who proved that when $\psi$ is monotonic, the Lebesgue measure of $W^{\times}_d(\psi)$ satisfies the following zero-one dichotomy:
\begin{equation*}\label{Gallagher}
\mathcal{L}^{d}\big(W^{\times}_d(\psi)\big)=\begin{cases}
0,  & \text{ if }\sum_{n=1}^\infty\psi(n)\left(\log\psi(n)^{-1}\right)^{d-1}<\infty,\\
1, & \text{ if } \sum_{n=1}^\infty \psi(n)\left(\log\psi(n)^{-1}\right)^{d-1}=\infty.
\end{cases}
\end{equation*}
Here $\mathcal{L}^{d}$ denotes the $d$-dimensional Lebesgue measure on $\R^d$. The Hausdorff theory of $W^{\times}_d(\psi)$ was developed in \cite{BoveyDodson78}, \cite{VBSV15} and \cite{HD18T}. Again, under the assumption that $\psi$ is monotonic, the Hausdorff measure of $W^{\times}_d(\psi)$ satisfies a zero-infinity dichotomy as follows: for $s\in (d-1,d)$, 
\begin{equation*}\label{HThMal}
\mathcal{H}^{s}\big(W^{\times}_d(\psi)\big)=\begin{cases}
0,  & \text{ if } \sum_{n=1}^\infty n^{d-s}\psi(n)^{s-d+1}<\infty,\\
\infty, & \text{ if }  \sum_{n=1}^\infty  n^{d-s}\psi(n)^{s-d+1}=\infty.
\end{cases}
\end{equation*}
Throughout, for a subset $A\subset\R^d$, we let $\mathcal{H}^s(A)$ denote the $s$-dimensional Hausdorff measure of $A$, and $\dim_{\rm H}A$ the Hausdorff dimension of $A$. See \cite{Falconer86T} for definitions and further details. For more results on the classical multiplicative Diophantine approximation, one refers to \cite{BVHAVS13}, \cite{BHV20S} and \cite{VBSV15}.

There also have been many works on multiplicative Diophantine approximation restricted to manifolds.  Badziahin and Levesley \cite{DBJL07} obtained convergence results for the Lebesgue measure and Hausdorff measure of the intersection of $W_{2}^{\times}(\psi)$ with a  non-degenerate $C^{(3)}$ planar curve $\mathcal{C}$, which were conjectured in \cite{VV06D}. A complete zero-infinity dichotomy for the Hausdorff measure of $W_{2}^{\times}(\psi)\cap \mathcal{C}$ was later obtained in \cite{VBSV15}.

In this paper, we investigate  sets in the framework of multiplicative Diophantine approximation. A key feature is that we consider approximation by rational vectors/numbers with  restricted denominators from two given sequences of positive integers.

Let $\psi:\N\to \R_{>0}$ be a positive function. Let $\mathcal{A}=\{a_n\}_{n\in\N}$, $\mathcal{B}=\{b_n\}_{n\in\N}$ be two sequences of positive integers. Set
\begin{equation}\label{eqDefWABT}
W_{\mathcal{A},\mathcal{B}}^{\times}(\psi)=\left\{(x,y)\in [0,1]^2: \|a_nx\|\|b_ny\|<\psi(n) \text{ for i.m. }n\in\N\right\},
\end{equation}
\begin{equation}\label{eqDefWABD}
W_{\mathcal{A},\mathcal{B}}(\psi)=\left\{x\in [0,1]: \|a_nx\|\|b_nx\|<\psi(n) \text{ for i.m. }n\in\N\right\}.
\end{equation}
It is clear that $W_{\mathcal{A},\mathcal{B}}(\psi)$ is a homothetic copy of the intersection of $W_{\mathcal{A},\mathcal{B}}^{\times}(\psi)$ and the diagonal of $[0,1]^2$. Throughout, we use the superscript $^\times$ in a set $E^{\times}$  to  indicate that $E^{\times}$ is a subset of $[0,1]^2$, the two-fold Cartesian product of $[0,1]$. The sets $W_{\mathcal{A},\mathcal{B}}^{\times}(\psi)$ and $W_{\mathcal{A},\mathcal{B}}(\psi)$ are our main objects of study in this paper. Before we state our results, we first describe some of our motivations.

Clearly, the set $W_2^{\times}(\psi)$ in the  classical multiplicative Diophantine approximation corresponds to $W_{\mathcal{A},\mathcal{B}}^{\times}(\psi)$ with $a_n=b_n=n$ for all $n$. For general sequences of positive integers $\mathcal{A}$ and $\mathcal{B}$, the set $W_{\mathcal{A},\mathcal{B}}^{\times}(\psi)$ can be naturally understood as multiplicative Diophantine approximation by rational vectors with  restricted denominators. In the one dimensional case, Diophantine approximation with restricted denominators has been intensively studied.
One refers to Chapter 6 of \cite{Harman98} for classical results and to \cite{PVZZ22I} for more recent developments. This is one reason for us to study $W_{\mathcal{A},\mathcal{B}}^{\times}(\psi)$, which is a natural analogue of approximation with restricted denominators in the multiplicative setting. 

 Another motivation for us to consider $W_{\mathcal{A},\mathcal{B}}^{\times}(\psi)$ with general sequences $\mathcal{A}$ and $\mathcal{B}$ comes from multiplicative Diophantine approximation in some dynamical settings. Recently, Li, Liao, Velani and Zorin \cite{LLVZ23A} extensively studied the shrinking target problem for matrix transformations of tori, which was initialed in \cite{HilVelani95T}.  Let $T={\rm diag}(t_1, t_2)$ be a diagonal integral matrix and  let $\psi: \N\to\R_{>0}$ be a positive function. Set
 \[W^{\times}(T,\psi)=\left\{x\in [0,1]^2:\|t_1^nx\|\|t_2^ny\|<\psi(n)\text{ for i.m. }n\in\N\right\}.\]
 As a special case of a more general result, the Hausdorff dimension of $W^{\times}(T,\psi)$ was obtained in \cite{LLVZ23A}. Notice that $W^{\times}(T,\psi)$ corresponds to $W^{\times}_{\mathcal{A},\mathcal{B}}(\psi)$ in our setting with $\mathcal{A}=\{t_1^n\}_{n\in\N}$ and $\mathcal{B}=\{t_2^n\}_{n\in\N}$.

 Our study of the set $W_{\mathcal{A},\mathcal{B}}(\psi)$ was motivated by multiplicative Diophantine approximation on planar curves. We focus on the intersection of $W^{\times}_{\mathcal{A},\mathcal{B}}(\psi)$ with the diagonal of $[0,1]^2$. This is different from \cite{DBJL07}, \cite{VV06D} and many other works in the literature, where the planar curve is often required to be non-degenerate.

 Now we introduce our results. In our first main result, we completely determine the Hausdorff dimensions of $W_{\mathcal{A},\mathcal{B}}^{\times}(\psi)$ and $W_{\mathcal{A},\mathcal{B}}(\psi)$ under the mild assumption that each of the sequences $\mathcal{A}$ and $\mathcal{B}$ consists of distinct elements. 

\begin{thm}\label{thm1Haus}
Let $\mathcal{A}=\{a_n\}_{n\in\N}$, $\mathcal{B}=\{b_n\}_{n\in\N}$ be two infinite subsets of positive integers and $\psi:\N\to (0,1)$ be a positive function. Let $W_{\mathcal{A},\mathcal{B}}^{\times}(\psi)$ and $W_{\mathcal{A},\mathcal{B}}(\psi)$ be defined as in \eqref{eqDefWABT} and \eqref{eqDefWABD}, respectively. 
Then we have  
 \[\dim_{\rm H}W_{\mathcal{A},\mathcal{B}}^{\times}(\psi)=\min\{1+\lambda,2\}\quad \text{ and }\quad \dim_{\rm H}W_{\mathcal{A},\mathcal{B}}(\psi)=\min\{\gamma,1\},\]
 where 
\begin{equation}\label{dimsstartimes}
\lambda=\inf\left\{s\geq 0: \sum_{n=1}^{\infty}\left[a_n\left(\frac{\psi(n)}{a_n}\right)^s+b_n\left(\frac{\psi(n)}{b_n}\right)^s\right]<\infty\right\},
\end{equation}
\begin{equation}\label{dimsstar}
\gamma=\inf\left\{s\geq 0: \sum_{n=1}^{\infty}\left[a_n\left(\frac{\psi(n)}{a_n}\right)^s+{\rm gcd}(a_n,b_n)\left(\frac{\psi(n)}{a_nb_n}\right)^{\frac{s}{2}}+b_n\left(\frac{\psi(n)}{b_n}\right)^s\right]<\infty\right\}.
\end{equation}
\end{thm}

Theorem \ref{thm1Haus} generalizes \cite[Theorem 9]{LLVZ23A} in the homogeneous case with $d=2$. Moreover, just like the one dimensional case (cf. \cite{KM20O}), when $\psi$ is not assumed to be monotonic, statements on $W^{\times}_{2}(\psi)$ (cf. \eqref{vb2}) can be reformulated as statements on $W^{\times}_{\mathcal{A},\mathcal{B}}(\psi)$  with $a_n=b_n$ for all $n$. Therefore, the first part of Theorem \ref{thm1Haus} also generalizes \cite[Corollary 4]{LFMH24T}.

The proof of Theorem \ref{thm1Haus} is partly based on the following general convergence result for the Hausdorff measures of $W_{\mathcal{A},\mathcal{B}}^{\times}(\psi)$ and $W_{\mathcal{A},\mathcal{B}}(\psi)$.

\begin{thm}\label{thm1HMeas}
Let $\mathcal{A}=\{a_n\}_{n\in\N}$, $\mathcal{B}=\{b_n\}_{n\in\N}$ be two sequences of positive integers and $\psi:\N\to (0,1)$ be a positive function. 
Let $s\in (0,1)$.
 Then  the following hold.
\begin{itemize}
\item[(i)]If 
\begin{equation*}\label{eqlcmqnvnqanbn}
\sum_{n=1}^{\infty}\left[a_n\left(\frac{\psi(n)}{a_n}\right)^s+b_n\left(\frac{\psi(n)}{b_n}\right)^s\right]<\infty,
\end{equation*}
then $\mathcal{H}^{1+s}(W_{\mathcal{A},\mathcal{B}}^{\times}(\psi))=0$.
\item[(ii)]If
\begin{equation*}\label{eqlcmqnvnqanbn}
\sum_{n=1}^{\infty}\left[a_n\left(\frac{\psi(n)}{a_n}\right)^s+{\rm gcd}(a_n,b_n)\left(\frac{\psi(n)}{a_nb_n}\right)^{\frac{s}{2}}+b_n\left(\frac{\psi(n)}{b_n}\right)^s\right]<\infty,
\end{equation*}
then $\mathcal{H}^{s}(W_{\mathcal{A},\mathcal{B}}(\psi))=0$. 
\end{itemize}
\end{thm}

Theorem \ref{thm1HMeas} (see also Theorem \ref{ThmNoNLacunary})  generalizes a recent result of L{\"u} and Zhang \cite{ZhangLu21}, who considered the special case that $\mathcal{A}=\{2^n\}_{n\in\N}$ and $\mathcal{B}=\{3^n\}_{n\in\N}$ and proved a zero-infinity dichotomy for the Hausdorff measure of $W_{\mathcal{A},\mathcal{B}}(\psi)$. 

As for the classical case, our setting of multiplicatively approximation is closely related to its simultaneous counterpart. Recently, Li, Liao, Velani and Zorin (\cite{MatrixManifold}, see also \cite[Remark 13]{LLVZ23A}) proved (partially conditioned on the validity of the {\em abc}-conjecture) that for $0\leq \tau\leq 1$, the set 
\[\left\{x\in [0,1]: \max\{\|2^nx\|,\|3^nx\|\}<3^{-n\tau}\text{ for i.m. }n\in\N\right\}\]
has Hausdorff dimension $\frac{1-\tau}{1+\tau}$. For general integers $b>a\geq 2$ beyond the case that $a=2$ and $b=3$, we find the Hausdorff dimension of the corresponding set  when $\tau>1$ (cf. Theorem \ref{ThmSabPsi}). This is also based on establishing a convergence result for the Hausdorff measure of the set in question (cf. \eqref{eqSimuCon}-\eqref{eqanlesbnspsi}).

\begin{thm}\label{ThmSabPsi}
Let $b>a\geq 2$ be positive integers and $\psi:\N\to(0,1)$ be a positive function. Set 
\[S_{a,b}(\psi)=\{x\in [0,1]: \max\{\|a^nx\|,\|b^nx\|\}<\psi(n)\text{ for i.m. }n\in\N\}.\]
If $\tau:=\varliminf_{n\to\infty}\frac{\log_b\psi(n)^{-1}}{n}>1$, then we have 
\[\dim_{\rm H}S_{a,b}(\psi)=\frac{\log_b{\rm gcd}(a,b)}{1+\tau}.\]
\end{thm}

The paper is organized as follows. In Section \ref{Sec2}, we give some preliminary lemmas about   sets which arise naturally in the definition of the limsup sets  that  we are concerned with. In Section \ref{Sec3},  we first prove Theorem \ref{thm1HMeas} and  then  apply it to deduce Theorem \ref{thm1Haus}. Theorem \ref{ThmSabPsi} is proved in Section \ref{Sec4}. Finally,  in the last section, we 
give some remarks concerning sharpness and generalizations of our results. 

\section{Preliminary lemmas}\label{Sec2}

We  first establish some preliminary lemmas which will be used in the proofs of our main results. To ease notation, we will use the Vinogradov symbol ``$\ll$'' to indicate an inequality with an unspecified positive multiplicative constant.

Given  positive integers $a,b$ and positive real numbers $ \varrho, \eta, \delta\in (0,1)$, set
 \[E_{a,b}(\varrho,\eta)=\left\{(x,y)\in [0,1]^2: \|ax\|<\varrho, \|by\|<\eta\right\}, \]
 \[E_{a,b}(\delta)=\left\{(x,y)\in [0,1]^2: \|ax\|\|by\|<\delta^2\right\},\]
 \[F_{a,b}(\varrho,\eta)=\left\{x\in [0,1]: \|ax\|<\varrho, \|bx\|<\eta\right\}.\]
 \[F_{a,b}(\delta)=\left\{x\in [0,1]: \|ax\|\|bx\|<\delta^2\right\}.\]
These sets arise naturally in the limsup sets that we are concerned with in this paper. For instance, for the sets defined in \eqref{eqDefWABT} and \eqref{eqDefWABD}, we have 
\[W_{\mathcal{A},\mathcal{B}}^{\times}(\psi)=\limsup_{n\to\infty}E_{a_n,b_n}\left(\psi(n)^{\frac{1}{2}}\right), \quad W_{\mathcal{A},\mathcal{B}}(\psi)=\limsup_{n\to\infty}F_{a_n,b_n}\left(\psi(n)^{\frac{1}{2}}\right).\]

 Our first lemma gives an upper bound for the Hausdorff content of the set $E_{a,b}(\varrho,\eta)$, which will be used in the proof of part (i) of Theorem \ref{thm1HMeas}.

 \begin{lem}\label{lemHconEnde}
For $0<s\leq 1$, we have 
\[\mathcal{H}^{1+s}_{\infty}(E_{a,b}(\varrho,\eta))\leq 64\varrho\eta\min\left\{\frac{\varrho}{a},\frac{\eta}{b}\right\}^{s-1}.\]
\end{lem}

\begin{proof}
Notice that 
\begin{align*}
E_{a,b}(\varrho,\eta)&=\bigcup_{k=0}^{a}\bigcup_{\ell=0}^{b}\left\{(x,y)\in [0,1]^2: |ax-k|<\varrho, |by-\ell|<\eta\right\}\\
&\subseteq \bigcup_{k=0}^{a}\bigcup_{\ell=0}^{b}\left[B\left(\frac{k}{a}, \frac{\varrho}{a}\right)\times B\left(\frac{\ell}{b}, \frac{\eta}{b}\right)\right].
\end{align*}
Here and afterwards, for $x\in\R$ and $r>0$, $B(x,r)$ denotes the open interval centered at $x$ of radius $r$. For each rectangular  in the above union, it can be covered by 
\[\frac{2\max\left\{\frac{\varrho}{a},\frac{\eta}{b}\right\}}{\min\left\{\frac{\varrho}{a},\frac{\eta}{b}\right\}}\]
many squares of side length $\min\left\{\frac{2\varrho}{a},\frac{2\eta}{b}\right\}$. Hence we have
\begin{align*}
\mathcal{H}^{1+s}_{\infty}(E_{a,b}(\varrho,\eta))&\leq (a+1)(b+1)\frac{2\max\left\{\frac{\varrho}{a},\frac{\eta}{b}\right\}}{\min\left\{\frac{\varrho}{a},\frac{\eta}{b}\right\}}\left(\sqrt{2}\min\left\{\frac{2\varrho}{a},\frac{2\eta}{b}\right\}\right)^{1+s}\\
&\leq 64ab\max\left\{\frac{\varrho}{a},\frac{\eta}{b}\right\}\min\left\{\frac{\varrho}{a},\frac{\eta}{b}\right\}^{s}\\
&=64\varrho\eta\min\left\{\frac{\varrho}{a},\frac{\eta}{b}\right\}^{s-1}.
\end{align*}
This completes the proof of the lemma.
\end{proof}

Based on Lemma \ref{lemHconEnde}, we have the following.

\begin{lem}\label{lemHconEabD}
Let $s\in (0,1)$. Then there exists a constant $C_1$ depending only on $s$ such that 
\[\mathcal{H}^{1+s}_{\infty}(E_{a,b}(\delta))\leq C_1\left[a\left(\frac{\delta^2}{a}\right)^s+b\left(\frac{\delta^2}{b}\right)^s\right].\]
\end{lem}

\begin{proof}
We decompose $E_{a,b}(\delta)$ as $E_{a,b}(\delta)=E_1\cup E_{2}\cup E_3$,
where 
\[E_1=\left\{(x,y)\in [0,1]^2: \|ax\|<\delta, \|by\|<\delta\right\},\]
\[E_2=\left\{(x,y)\in [0,1]^2: \|ax\|\geq \delta, \|ax\|\|by\|<\delta^2\right\},\]
\[E_3=\left\{(x,y)\in [0,1]^2: \|by\|\geq \delta, \|ax\|\|by\|<\delta^2\right\}.\]
By Lemma \ref{lemHconEnde}, we see that 
\begin{align*}
\mathcal{H}^{1+s}_{\infty}(E_1)&\leq 64\delta^2\min\left\{\frac{\delta}{a},\frac{\delta}{b}\right\}^{s-1}\\
&=64\delta^{1+s}\max\{a,b\}^{1-s}.	
\end{align*}
Note that 
\begin{align*}
E_2&=\bigcup_{j\geq 0:2^{j+1}\delta\leq 1}\left\{(x,y)\in [0,1]^2: 2^j\delta\leq \|ax\|<2^{j+1}\delta, \|ax\|\|by\|<\delta^2\right\}\\
&\subseteq\bigcup_{j\geq 0:2^{j+1}\delta\leq 1}\left\{(x,y)\in [0,1]^2: \|ax\|<2^{j+1}\delta, \|by\|<2^{-j}\delta\right\}.
\end{align*}
Apply Lemma \ref{lemHconEnde} to each set in the above union, we have  
\begin{align*}
\mathcal{H}^{1+s}_{\infty}(E_2)&\leq \sum_{j\geq 0: 2^{j+1}\delta\leq 1}128\delta^2\min\left\{\frac{2^{j+1}\delta}{a},\frac{2^{-j}\delta}{b}\right\}^{s-1}\\
&\ll \delta^{1+s}\left[a^{1-s}\sum_{j\geq0: 2^{j+1}\delta\leq 1}2^{(j+1)(s-1)}+b^{1-s}\sum_{j\geq 0: 2^{j+1}\delta\leq 1} 2^{j(1-s)}\right]\\
&\ll \delta^{1+s}\left[a^{1-s}+b^{1-s}\delta^{s-1}\right]\\
&=\delta^{1+s}a^{1-s}+\delta^{2s}b^{1-s}.
\end{align*}
Here the constant in each ``$\ll$'' depends only on $s$.
To estimate $\mathcal{H}^{1+s}_{\infty}(E_3)$, note that 
\begin{align*}
E_3&=\bigcup_{j\geq0: 2^{j+1}\delta\leq 1}\left\{(x,y)\in [0,1]^2: 2^j\delta\leq \|by\|<2^{j+1}\delta, \|ax\|\|by\|<\delta^2\right\}\\
&\subseteq \bigcup_{j\geq 0: 2^{j+1}\delta\leq 1}\left\{(x,y)\in [0,1]^2: \|by\|<2^{j+1}\delta, \|ax\|<2^{-j}\delta\right\}.
\end{align*}
Then it follows from a similar argument as for $\mathcal{H}^{1+s}(E_2)$ that 
\[\mathcal{H}^{1+s}_{\infty}(E_3)\ll \delta^{1+s}b^{1-s}+\delta^{2s}a^{1-s}.\]

Combining the above estimates together, we obtain that 
\begin{align*}
\mathcal{H}^{1+s}_{\infty}(E_{a,b}(\delta))&\ll \delta^{1+s}\max\{a,b\}^{1-s}+\delta^{1+s}a^{1-s}+\delta^{2s}b^{1-s}+\delta^{1+s}b^{1-s}+\delta^{2s}a^{1-s}\\
&\ll \delta^{2s}(a^{1-s}+b^{1-s}),
\end{align*}
where the last inequality holds since $0<s<1$ and $\delta\in (0,1)$.
\end{proof}

Concerning the set $F_{a,b}(\varrho,\eta)$, we prove the following covering property, which is needed in the proof of part (ii) of Theorem \ref{thm1HMeas}.
\begin{lem}\label{corcovering}
The set $F_{a,b}(\varrho,\eta)$ can be covered by at most 
\[12\left[{\rm gcd}(a,b)+2ab\max\left\{\frac{\varrho}{a},\frac{\eta}{b}\right\}\right]\]
many intervals of length $2\min\left\{\frac{\varrho}{a},\frac{\eta}{b}\right\}$.
\end{lem}

In the proof of Lemma \ref{corcovering}, we will make use of the following version of Erd\H{o}s-Tur{\'a}n inequality. Let $\{u_n\}_{n=1}^N$ be a sequence of  $N$ real numbers.  Let $\alpha,\beta\in \R$ with $\alpha< \beta< \alpha+1$. The discrepancy of  $\{u_n\}_{n=1}^N$ is defined by
\[D(N;\alpha,\beta)=\#\{1\leq n\leq N: u_n\in(\alpha,\beta) (\bmod1)\}-(\beta-\alpha)N.\]
Erd\H{o}s-Tur{\'a}n inequality gives a very useful upper bound for the discrepancy.
\begin{lem}\label{eqETinq}\cite[Chapter 1, Theorem 1]{Montgomery94}
For each $K\in\N$, 
\[|D(N;\alpha,\beta)|\leq \frac{N}{K+1}+2\sum_{k=1}^Kc_k\left|\sum_{n=1}^Ne^{2\pi i u_n k}\right|,\]
where 
\[c_k=\frac{1}{K+1}+\min\left\{\beta-\alpha, \frac{1}{\pi k}\right\}.\]
\end{lem}

\begin{proof}[Proof of Lemma \ref{corcovering}]
Notice that 
\begin{align*}
F_{a,b}(\varrho,\eta)&=\bigcup_{k=0}^{a}\bigcup_{\ell=0}^{b}\{x\in [0,1]: |ax-k|<\varrho, |bx-\ell|<\eta\}\\
&\subseteq \bigcup_{k=0}^{a}\bigcup_{\ell=0}^{b}\left[B\left(\frac{k}{a}, \frac{\varrho}{a}\right)\cap B\left(\frac{\ell}{b}, \frac{\eta}{b}\right)\right].
\end{align*}
Let $N_{a,b}(\varrho,\eta)$ be the number of pairs $(k,\ell)\in\Z^2$ with $0\leq k\leq a$ and $0\leq \ell\leq b$ such that 
\begin{equation}\label{eqintsein}
B\left(\frac{k}{a}, \frac{\varrho}{a}\right)\cap B\left(\frac{\ell}{b}, \frac{\eta}{b}\right)\neq \emptyset.	
\end{equation}
Since each set in \eqref{eqintsein} is an interval of length at most $2\min\left\{\frac{\varrho}{a},\frac{\eta}{b}\right\}$, hence to prove the lemma it suffices to prove that 
\begin{equation}\label{eqNabde}
N_{a,b}(\varrho,\eta)\leq 12\left[{\rm gcd}(a,b)+2ab\max\left\{\frac{\varrho}{a},\frac{\eta}{b}\right\}\right].
\end{equation}
When $a=b$, since $\varrho,\eta\in (0,1)$, we see that for each $0\leq k\leq a$ fixed, there exist at most three $\ell$'s satisfying \eqref{eqintsein}. Hence  $N_{a,b}(\varrho,\eta)\leq 3(a+1)$ and so the lemma holds. In the following, we assume that $a<b$; the other case that $a>b$ can be proved similarly.

Let $\theta$ be such that 
\begin{equation*}\label{eqtildedelta}
\frac{\theta}{a}=\frac{\varrho}{a}+\frac{\eta}{b}.
\end{equation*}
Then \eqref{eqintsein} holds if and only if 
\begin{equation}\label{eqanbninh}
\left|\frac{k}{a}-\frac{\ell}{b}\right|<\frac{\theta}{a}.
\end{equation}
Hence 
\begin{align*}
N_{a,b}(\varrho,\eta)=\#\left\{(k,\ell)\in \Z^2: 0\leq k\leq a, 0\leq \ell\leq b \text{ such that }\eqref{eqanbninh} \text{ holds} \right\}. 
\end{align*}
Let $g={\rm gcd}(a,b)$, $a'=\frac{a}{g}$ and $b'=\frac{b}{g}$. Then $b'>1$ and ${\rm gcd}(a',b')=1$.
Below we estimate $N_{a,b}(\varrho,\eta)$ in the three scenarios $\theta\leq \frac{1}{b'}$, $\theta\geq \frac{1}{2}$ and $\frac{1}{b'}<\theta< \frac{1}{2}$, separately.
  
Notice that 
\begin{equation}\label{eqkxinan}
\left|\frac{k}{a}-\frac{\ell}{b}\right|=\left|\frac{kb-\ell a}{ab}\right|=\left|\frac{kb'-\ell a'}{ab'}\right|.
\end{equation}
Hence if $\theta\leq \frac{1}{b'}$, then \eqref{eqanbninh} and \eqref{eqkxinan}  imply that $kb'=\ell a'$.
Since ${\rm gcd}(a', b')=1$, this further implies that $k=ta'$,  $\ell=tb'$ for some $t\in \Z$. Since $0\leq k\leq a$ and $0\leq \ell\leq b$, each such $t$  satisfies that $0\leq t\leq g$.
Therefore, when  $\theta\leq \frac{1}{b'}$, we have  
\begin{equation}\label{eqNnleq22g}
N_{a,b}(\varrho,\eta)\leq 1+g\leq 2g.
\end{equation}

Next we consider the case that $\theta\geq \frac{1}{2}$. This time,  for $k,\ell$ satisfying \eqref{eqanbninh}, we have 
\[\left| k-\frac{a\ell}{b}\right|<2\theta.\]
So for each $0\leq \ell\leq b$, there are at most $\lfloor2\theta\rfloor+2$ many $k$'s satisfying \eqref{eqanbninh}. Hence,
\begin{equation}\label{eqcasethe12}
N_{a,b}(\varrho,\eta)\leq (b+1)(\lfloor2\theta\rfloor+2)\leq 12b.
\end{equation}

Finally, we consider the case that $\frac{1}{b'}<\theta< \frac{1}{2}$. Since $\theta<\frac{1}{2}$,  for $k,\ell$ satisfying \eqref{eqanbninh}, we have 
\[\left| k-\frac{a\ell}{b}\right|<\frac{1}{2} \quad \text{ and }\quad \left\|\frac{a\ell}{b}\right\|<\theta.\]
Hence for each $\ell$ with $\left\|\frac{a\ell}{b}\right\|<\theta$ there is at most one $k$ satisfying \eqref{eqanbninh}.
Therefore,
\begin{align}
N_{a,b}(\varrho,\eta)&\leq \#\left\{0\leq \ell \leq b: \left\|\frac{a\ell}{b}\right\|<\theta\right\}\nonumber\\
&=\sum_{\ell=0}^{b}\chi_{\theta}\left(\frac{a\ell}{b}\right), \label{eqNntilded}
\end{align}
where $\chi_{\theta}(\cdot)$ denotes the characteristic function for the set $\left\{x\in \R: \|x\|<\theta\right\}$. Below we apply the Erd\H{o}s-Tur{\'a}n inequality (Lemma \ref{eqETinq}) to estimate \eqref{eqNntilded}. To this end, in  Lemma \ref{eqETinq} we take $\alpha=-\theta$, $\beta=\theta$, $N=b+1$,  $K=\lfloor \theta^{-1}\rfloor$, and $u_{\ell}=\frac{a\ell}{b}$ for $\ell=0, 1,\ldots, b$. Notice that $K<b'$, $\frac{1}{K+1}<\theta$, and for $k=1,\ldots, K$,
	\[c_k=\frac{1}{K+1}+\min\left\{\beta-\alpha, \frac{1}{\pi k}\right\}\leq 3\theta.\]
Let \[D(b+1;-\theta,\theta)=\sum_{\ell=0}^{b}\chi_{\theta}\left(\frac{a\ell}{b}\right)-2(b+1)\theta.\]
Then by Lemma \ref{eqETinq}, we have 
\begin{align}
|D(b+1;-\theta,\theta)|&\leq (b+1)\theta+6\theta\sum_{k=1}^K\left|\sum_{\ell=0}^{b}e^{2\pi i k \frac{a\ell}{b}}\right|\nonumber\\
&=(b+1)\theta+6\theta\sum_{k=1}^K\left|\sum_{\ell=0}^{b}e^{2\pi i k \frac{a'\ell}{b'}}\right|\nonumber\\
&= (b+1)\theta+6K\theta,\nonumber
\end{align}
where the last equality holds since $b'>1$ and $(a',b')=1$ and so $\sum_{\ell=0}^{b}e^{2\pi i k \frac{a'\ell}{b'}}=1$. Since $K<b'$, it follows that 
\[|D(b+1;-\theta,\theta)|\leq (b+1)\theta+6b'\theta\leq 8b\theta.\]
Therefore, we have 
\begin{equation*}\label{eqNndeltadleq}
N_{a,b}(\varrho,\eta)\leq |D(b+1;-\theta, \theta)|+2(b+1)\theta\leq 12b\theta.
\end{equation*}
This combining with \eqref{eqNnleq22g}-\eqref{eqcasethe12} yields that 
\[N_{a,b}(\varrho,\eta)\leq 12(g+b\theta).\]
Since 
\[b\theta=b\varrho+a\eta\leq 2ab\max\left\{\frac{\varrho}{a},\frac{\eta}{b}\right\},\]
we see that \eqref{eqNabde} holds and we complete the proof of the lemma.
\end{proof}

Recall that for $a,b\in \N$ and $\delta\in (0,1)$, $$F_{a,b}(\delta)=\left\{x\in [0,1]: \|ax\|\|bx\|<\delta^2\right\}.$$
Below we apply Lemma \ref{corcovering} to establish  an upper bound estimate for the Hausdorff content of the set $F_{a,b}(\delta)$, which is the key to prove part (ii) of Theorem \ref{thm1HMeas}.

\begin{lem}\label{lemHsinftyee}
Let $s\in (0,1)$. Then there exists a constant $C_2$ depending only on $s$ such that 
\[\mathcal{H}^s_{\infty}(F_{a,b}(\delta))\leq C_2\left[a\left(\frac{\delta^2}{a}\right)^s+{\rm gcd}(a,b)\left(\frac{\delta}{\sqrt{ab}}\right)^s+b\left(\frac{\delta^2}{b}\right)^s\right].\]
\end{lem}

\begin{proof}
We decompose $F_{a,b}(\delta)$ as $F_{a,b}(\delta)=F_1\cup F_2\cup F_3$,
where 
\[F_1=\{x\in [0,1]: \|ax\|<\delta, \|bx\|<\delta\},\]
\[F_2=\left\{x\in [0,1]: \|ax\|\geq \delta, \|ax\|\|bx\|<\delta^2\right\},\]
\[F_3=\left\{x\in [0,1]: \|bx\|\geq \delta, \|ax\|\|bx\|<\delta^2\right\}.\]
In the following, we estimate the $s$-dimensional Hausdorff content of $F_i$ ($i=1,2,3$) separately. To ease notation, let $g={\rm gcd}(a,b)$.

{\bf An upper bound for $\mathcal{H}^s_{\infty}(F_1)$}. By Lemma \ref{corcovering},  we have 
\begin{align}\label{eqHsWF1}
\mathcal{H}^s_{\infty}(F_1)&\leq 12\left(g+2ab\max\left\{\frac{\delta}{a},\frac{\delta}{b}\right\}\right)\times \left(2\min\left\{\frac{\delta}{a}, \frac{\delta}{b}\right\}\right)^s\\
&\ll\left(g+\frac{ab\delta}{\min\{a,b\}}\right)\frac{\delta^s}{\max\{a,b\}^s},\nonumber
\end{align}
where the constant in $\ll$ depends only on $s$ and is independent of $a,b$ and $\delta$.

{\bf An  upper bound for $\mathcal{H}_{\infty}^{s}(F_2)$}.  Let $J=\{j\geq 0: 2^{j+1}\delta<1\}$, $J_1=\{j\geq 0: 2^{2j+1}\leq a/b\}$ and $J_2=\{j\geq 0: 2^{2j+1}> a/b\}$. Note that 
\begin{align*}
F_2&=\bigcup_{j\in J}\left\{x\in [0,1]: 2^j\delta\leq \|ax\|<2^{j+1}\delta, \|ax\|\|bx\|<\delta^2\right\}\\
&\subseteq \bigcup_{j\in J}\left\{x\in [0,1]: \|ax\|<2^{j+1}\delta, \|bx\|<2^{-j}\delta\right\}.
\end{align*}
Apply Lemma \ref{corcovering} to each set in the above union, we have
\begin{align}
\mathcal{H}_{\infty}^s(F_2)&\leq \sum_{j\in J}12\left(g+2ab\max\left\{\frac{2^{j+1}\delta}{a},\frac{2^{-j}\delta}{b}\right\}\right)\times \left(2\min\left\{\frac{2^{j+1}\delta}{a},\frac{2^{-j}\delta}{b}\right\}\right)^s\nonumber\\
\begin{split}
&\leq \sum_{j\in J_1} 12\left(g+2a\cdot 2^{-j}\delta\right)\times \left(2\cdot\frac{2^{j+1}\delta}{a}\right)^s \\ &\qquad+\sum_{j\in J\cap J_2}  12\left(g+2b\cdot2^{j+1}\delta\right)\times \left(2\cdot\frac{2^{-j}\delta}{b}\right)^s
\end{split}\nonumber\\
\begin{split}
&\ll g\left(\frac{\delta}{a}\right)^s\sum_{j\in J_1}2^{js}+a\delta\left(\frac{\delta}{a}\right)^s\sum_{j\in J_1} 2^{(s-1)j} \\ &\qquad+g\left(\frac{\delta}{b}\right)^s\sum_{j\in J\cap J_2}2^{-js}+b\delta\left(\frac{\delta}{b}\right)^s\sum_{j\in J\cap J_2}2^{(1-s)j}
\end{split}\label{eqHasMF2}
\\
&\ll  g\left(\frac{\delta}{a}\right)^s\left(\frac{a}{b}\right)^{\frac{s}{2}}+a\delta\left(\frac{\delta}{a}\right)^s+g\left(\frac{\delta}{b}\right)^s\left(\frac{b}{a}\right)^{\frac{s}{2}}+b\delta\left(\frac{\delta}{b}\right)^s\delta^{s-1}\nonumber\\
&\ll g\left(\frac{\delta}{\sqrt{ab}}\right)^s+a\delta\left(\frac{\delta}{a}\right)^s+b\left(\frac{\delta^2}{b}\right)^s.\nonumber
\end{align}
Again the constant in each of the above ``$\ll$'' depends only on $s$.

{\bf An upper bound for $\mathcal{H}_{\infty}^{s}(F_3)$}. Similar to the case for $F_2$, we have 
\begin{align*}
F_3&=\bigcup_{j\in J}\left\{x\in [0,1]: 2^j\delta\leq \|bx\|<2^{j+1}\delta, \|ax\|\|bx\|<\delta^2\right\}\\
&\subseteq \bigcup_{j\in J}\left\{x\in [0,1]: \|bx\|<2^{j+1}\delta, \|ax\|<2^{-j}\delta\right\}.
\end{align*}
Then a similar argument as above yields that 
\begin{align*}
\mathcal{H}_{\infty}^s(F_3)\ll g\left(\frac{\delta}{\sqrt{ab}}\right)^s+b\delta\left(\frac{\delta}{b}\right)^s+a\left(\frac{\delta^2}{a}\right)^s.
\end{align*}

Combining the above upper bounds for $\mathcal{H}^s_{\infty}(F_i)$ ($i=1,2,3$), we obtain that 
\begin{align*}
\begin{split}
\mathcal{H}_{\infty}^s(F_{a,b}(\delta))&\ll\left(g+\frac{ab\delta}{\min\{a,b\}}\right)\frac{\delta^s}{\max\{a,b\}^s}+g\left(\frac{\delta}{\sqrt{ab}}\right)^s+a\delta\left(\frac{\delta}{a}\right)^s+b\left(\frac{\delta^2}{b}\right)^s \\ &\qquad +g\left(\frac{\delta}{\sqrt{ab}}\right)^s+b\delta\left(\frac{\delta}{b}\right)^s+a\left(\frac{\delta^2}{a}\right)^s
\end{split}\\
&\ll a\left(\frac{\delta^2}{a}\right)^s+g\left(\frac{\delta}{\sqrt{ab}}\right)^s+b\left(\frac{\delta^2}{b}\right)^s,
\end{align*}
where in the last $\ll$ we have used the assumption that $s\in (0,1)$. This completes the proof of the lemma. 
\end{proof}

Concerning the Lebesgue measure of $F_{a,b}(\delta)$, we have the following result.

\begin{lem}\label{lemLebesgue}
There exists an absolute constant $C_3$ such that 
\[\mathcal{L}(F_{a,b}(\delta))\leq C_3\left[{\rm gcd}(a,b)\left(\frac{\delta^2}{ab}\right)^{\frac{1}{2}}+\delta^2\log\left(\frac{1}{\delta}\right)\right].\]
\end{lem}

\begin{proof}
The proof is a slight modification of that of Lemma \ref{lemHsinftyee}, since most part of the proof of Lemma \ref{lemHsinftyee} still works when $s=1$. To see this, let $F_i$ $(i=1, 2, 3)$ be given as in Lemma $\ref{lemHsinftyee}$ so that $F_{a,b}(\delta)=F_1\cup F_2\cup F_3$. Then letting $s=1$ in \eqref{eqHsWF1} yields that  
\[\mathcal{L}(F_1)\leq 48\left(\frac{g\delta}{\max\{a,b\}}+\delta^2\right),\]
where $g={\rm gcd}(a,b)$. As for $F_2$, notice that when $s=1$, the third line in the estimate of $\mathcal{H}^s_{\infty}(F_2)$ (cf. \eqref{eqHasMF2})  becomes 
\begin{align*}
\mathcal{L}(F_2)&\ll \frac{g\delta}{a}\sum_{j\in J_1}2^j+\frac{g\delta}{b}\sum_{j\in J\cap J_2}2^{-j}+(\#J)\delta^2\\
&\ll \frac{g\delta}{a}\left(\frac{a}{b}\right)^{\frac{1}{2}}+\frac{g\delta}{b}\left(\frac{b}{a}\right)^{\frac{1}{2}}+\delta^2\log\left(\frac{1}{\delta}\right)\\
&\ll \frac{g\delta}{\sqrt{ab}}+\delta^2\log\left(\frac{1}{\delta}\right),
\end{align*}
where all constants in ``$\ll$'' are absolute. By a a similar argument, the same bound holds for $\mathcal{L}(F_3)$. Hence the lemma  follows by combining these upper bounds for $\mathcal{L}(F_i)$, $i=1,2,3$.
\end{proof}

The following simple observation plays an important role in  our treatment of $W_{\mathcal{A},\mathcal{B}}(\psi)$ (cf. \eqref{eqDefWABD}).  

\begin{lem}\label{LemSubset}
Let $a,b$ be positive integers and let $g={\rm gcd}(a,b)$. Then for  $\eta\in (0,1)$, we have 
\[\{x\in [0,1]: \|gx\|<g\eta\}\subseteq \{x\in [0,1]: \|ax\|<a\eta,\|bx\|<b\eta\}.\]
\end{lem}

\begin{proof}
According to the definition of $\|\cdot\|$, for any $x,y\in\R$, we have 
\[\|x+y\|\leq \|x\|+\|y\|.\]
As a consequence, for any $x\in \R$ and every $n\in\N$,
\[\|nx\|\leq n\|x\|.\]
Therefore, for any $x\in [0,1]$ with $\|gx\|<g\eta$, we have 
\[\|ax\|=\left\|\frac{a}{g}gx\right\|\leq \frac{a}{g}\|gx\|<a\eta,\quad \|bx\|=\left\|\frac{b}{g}gx\right\|\leq \frac{b}{g}\|gx\|<b\eta.\]
From this the lemma follows. 
\end{proof}

\section{Proof of Theorems \ref{thm1Haus}-\ref{thm1HMeas}}\label{Sec3}

We first give the proof of Theorem \ref{thm1HMeas}. Then we apply Theorem \ref{thm1HMeas} to deduce Theorem \ref{thm1Haus}.

\begin{proof}[Proof of Theorem \ref{thm1HMeas}]
For $n\in\N$, let $$E_n=\left\{(x,y)\in [0,1]^2: \|a_nx\|\|b_ny\|<\psi(n)\right\}.$$
Then $W_{\mathcal{A},\mathcal{B}}^{\times}(\psi)=\limsup_{n\to\infty}E_n$. By Lemma \ref{lemHconEabD},  there is a constant $C_1$ which is independent of $n$ such that
\[\mathcal{H}^{1+s}_{\infty}(E_n)\leq C_1\left[a_n\left(\frac{\psi(n)}{a_n}\right)^s+b_n\left(\frac{\psi(n)}{b_n}\right)^s\right].\]
It then follows from the Borel-Cantelli lemma that $\mathcal{H}^{1+s}_{\infty}(W_{\mathcal{A},\mathcal{B}}^{\times}(\psi))=0$ and thus  $\mathcal{H}^{1+s}(W_{\mathcal{A},\mathcal{B}}^{\times}(\psi))=0$. This proves part (i) of Theorem \ref{thm1HMeas}. 

Since $W_{\mathcal{A},\mathcal{B}}(\psi)=\limsup_{n\to\infty}F_n$, where 
\begin{equation*}\label{eqDefFn}
F_n=\{x\in [0,1]: \|a_nx\|\|b_nx\|<\psi(n)\},
\end{equation*}
 the part (ii) of Theorem \ref{thm1HMeas} follows similarly by applying Lemma \ref{lemHsinftyee} and again the Borel-Cantelli lemma.
\end{proof}

As a direct consequence of  Theorem \ref{thm1HMeas}, we see that the Hausdorff dimension of $W_{\mathcal{A},\mathcal{B}}^{\times}(\psi)$ is bounded above by $\min\{1+\lambda,2\}$, and that of  $W_{\mathcal{A},\mathcal{B}}(\psi)$ is bounded above by $\min\{\gamma,1\}$, where $\lambda$ and $\gamma$ are defined in \eqref{dimsstartimes} and \eqref{dimsstar}, respectively. To show that these are also lower bounds, we need make use of a result  about Hausdorff dimension in one dimensional Diophantine approximation.

Let $\phi:\N\to \R_{\geq0}$ be a nonnegative function. Set 
\[W(\phi)=\{x\in [0,1]: \|qx\|<\phi(q) \text{ for i.m. }q\in \N\}.\]
The size of $W(\phi)$  is a core subject of study in metric Diophantine approximation. Under the assumption that $\phi$ is monotonically non-increasing, Jarn{\'i}k \cite{Jarnik31} proved that  the Hausdorff measure of $W(\phi)$ satisfies a zero-full law according to a series converges or diverges, and so obtained the Hausdorff dimension of  $W(\phi)$. For general $\phi$ without the monotonicity assumption, Hinokuma and Shiga \cite{HS96H} found a formula for the Hausdorff dimension of $W(\phi)$. Later, Rynne \cite{Rynne98T} observed that the dimension formula obtained in \cite{HS96H} can be simplified as follows.

\begin{thm}\cite{Rynne98T}\label{thm2KM}
Let $\phi:\N\to \R_{\geq0}$ be a nonnegative function.
Then 
$$\dim_{\rm H}W(\phi)=\min\{s_*,1\},$$ where 
\begin{equation*}\label{eqlAonly}
s_*=\inf\left\{s\geq 0: \sum_{q=1}^{\infty}q\left(\frac{\phi(q)}{q}\right)^s<\infty\right\}.
\end{equation*}
\end{thm}

Since Theorem \ref{thm2KM} holds without assuming $\phi$ is monotonic, it can applied to deduce the following result on approximation for subsequences of integers, which will be used in our proof of Theorem \ref{thm1Haus}.

\begin{lem}\label{lemHanM}
 Let $\mathcal{A}=\{a_n\}_{n\in\N}$ be an infinite subset of positive integers and $\psi: \N\to (0,1)$ be a positive function. Set 
\[W_{\mathcal{A}}(\psi)=\{x\in [0,1]: \|a_nx\|<\psi(n) \text{ for i.m. } n\in\N\}.\]
Then  we have $$\dim_{\rm H}W_{\mathcal{A}}(\psi)=\min\{s_*,1\},$$ where 
\begin{equation*}\label{eqlAonly}
s_*=\inf\left\{s\geq 0: \sum_{n=1}^{\infty}a_n\left(\frac{\psi(n)}{a_n}\right)^s<\infty\right\}.
\end{equation*}
\end{lem}

To see Lemma \ref{lemHanM}, define a nonnegative function $\phi: \N\to \R_{\geq0}$ by 
\[\phi(q):=\begin{cases}\psi(n), & \text{ if }q=a_n\in\mathcal{A},\\
0, & \text{ if }q\not\in \mathcal{A}.
\end{cases}\]
Notice that $\phi$ is well-defined since the elements in the sequence $\mathcal{A}$ are distinct. 
Clearly,  $W_{\mathcal{A}}(\psi)=W(\phi)=\{x\in [0,1]: \|qx\|<\phi(q) \text{ for i.m. }q\in \N\}$. Then Lemma \ref{lemHanM} readily follows from Theorem \ref{thm2KM}.

\begin{proof}[Proof of Theorem \ref{thm1Haus}]
\textbf{Hausdorff dimension of $W_{\mathcal{A},\mathcal{B}}^{\times}(\psi)$.} Recall that 
\begin{equation*}\label{dimlambda}
\lambda=\inf\left\{s\geq 0: \sum_{n=1}^{\infty}\left[a_n\left(\frac{\psi(n)}{a_n}\right)^s+b_n\left(\frac{\psi(n)}{b_n}\right)^s\right]<\infty\right\}.
\end{equation*}
We aim to show that 
\begin{equation}\label{eqDimForWAB}
W_{\mathcal{A},\mathcal{B}}^{\times}(\psi)=\min\{1+\lambda,2\}.	
\end{equation}
 To see the ``$\leq$'' part, we may assume that $\lambda<1$ since otherwise there is nothing to prove. Then for any $s\in (\lambda,1)$ we have 
\[\sum_{n=1}^{\infty}\left[a_n\left(\frac{\psi(n)}{a_n}\right)^s+b_n\left(\frac{\psi(n)}{b_n}\right)^s\right]<\infty.\]
It then follows from Theorem \ref{thm1HMeas} that $\mathcal{H}^{1+s}(W_{\mathcal{A},\mathcal{B}}^{\times}(\psi))=0$. Hence $\dim_{\rm H}W_{\mathcal{A},\mathcal{B}}^{\times}(\psi)\leq 1+\lambda$.
To prove the ``$\geq$'' part, 
notice that 
\[W_{\mathcal{A},\mathcal{B}}^{\times}(\psi)\supseteq W_{\mathcal{A}}(\psi)\times[0,1] \quad \text{ and }\quad W_{\mathcal{A},\mathcal{B}}^{\times}(\psi)\supseteq [0,1] \times W_{\mathcal{B}}(\psi),\]
where 
\[W_{\mathcal{A}}(\psi)=\{x\in [0,1]: \|a_nx\|<\psi(n)\text{ for i.m. }n\in\N\},\]
\[W_{\mathcal{B}}(\psi)=\{y\in [0,1]: \|b_ny\|<\psi(n)\text{ for i.m. }n\in\N\}.\]
Hence we have 
\begin{equation}\label{eqDimW01P}
\dim_{\rm H}W_{\mathcal{A},\mathcal{B}}^{\times}(\psi)\geq 1+\max\{\dim_{\rm H}W_{\mathcal{A}}(\psi), \dim_{\rm H}W_{\mathcal{B}}(\psi)\}.
\end{equation}
Here we have used a well-known inequality for the Hausdorff dimension of Cartesian products of sets; see e.g. \cite[Corollary 5.10]{Falconer86T}.
We may assume that $\lambda>0$, since otherwise the ``$\geq$'' part of \eqref{eqDimForWAB} holds trivially. Then for any $s\in (0,\min\{\lambda,1\})$ we have 
\[\sum_{n=1}^{\infty}\left[a_n\left(\frac{\psi(n)}{a_n}\right)^s+b_n\left(\frac{\psi(n)}{b_n}\right)^s\right]=\infty,\]
which implies that either $\sum_{n=1}^{\infty}a_n\left(\frac{\psi(n)}{a_n}\right)^s$ or 
$\sum_{n=1}^{\infty}b_n\left(\frac{\psi(n)}{b_n}\right)^s$ diverges. Hence by Lemma \ref{lemHanM}, either $W_{\mathcal{A}}(\psi)$ or $W_{\mathcal{B}}(\psi)$ has Hausdorff dimension at least $\min\{1,\lambda\}$. This combining with \eqref{eqDimW01P} yields the ``$\geq$'' part of \eqref{eqDimForWAB}.

\textbf{Hausdorff dimension of $W_{\mathcal{A},\mathcal{B}}(\psi)$.} To prove that 
\[\dim_{\rm H}W_{\mathcal{A},\mathcal{B}}(\psi)=\min\{\gamma,1\},\]
where 
\begin{equation}\label{dimgamma}
\gamma=\inf\left\{s\geq 0: \sum_{n=1}^{\infty}\left[a_n\left(\frac{\psi(n)}{a_n}\right)^s+{\rm gcd}(a_n,b_n)\left(\frac{\psi(n)}{a_nb_n}\right)^{\frac{s}{2}}+b_n\left(\frac{\psi(n)}{b_n}\right)^s\right]<\infty\right\},
\end{equation}
we first notice that  a similar argument as above easily yields that $\dim_{\rm H}W_{\mathcal{A},\mathcal{B}}(\psi)\leq\min\{\gamma,1\}$. Below we prove the ``$\geq$'' part.

Again we may assume that $\gamma>0$. Then for any $s\in (0,\min\{\gamma,1\})$, the series in \eqref{dimgamma} diverges and so at least one of the following three series diverges: 
\begin{equation}\label{threeseries}
\sum_{n=1}^{\infty}a_n\left(\frac{\psi(n)}{a_n}\right)^s, \quad \sum_{n=1}^{\infty}g_n\left(\frac{\psi(n)}{a_nb_n}\right)^{\frac{s}{2}}, \quad \sum_{n=1}^{\infty}b_n\left(\frac{\psi(n)}{b_n}\right)^s,	
\end{equation}
where $g_n={\rm gcd}(a_n,b_n)$.

It is clear that 
\[W_{\mathcal{A},\mathcal{B}}(\psi)\supseteq \left\{x\in [0,1]: \|a_nx\|<\psi(n)\text{ for i.m. }n\in\N\right\}=W_{\mathcal{A}}(\psi),\] 
\[W_{\mathcal{A},\mathcal{B}}(\psi)\supseteq \left\{x\in [0,1]: \|b_nx\|<\psi(n)\text{ for i.m. }n\in\N\right\}=W_{\mathcal{B}}(\psi).\]
Furthermore, by Lemma \ref{LemSubset}, we have 
\begin{equation}\label{eqWG}
W_{\mathcal{A},\mathcal{B}}(\psi)\supseteq \left\{x\in [0,1]: \|g_nx\|<g_n\left(\frac{\psi(n)}{a_nb_n}\right)^{\frac{1}{2}} \text{ for i.m. }n\in\N\right\}.	
\end{equation}
By \eqref{threeseries} and  Lemma \ref{lemHanM}, one of $W_{\mathcal{A}}(\psi), W_{\mathcal{B}}(\psi)$, and  the right-hand side of \eqref{eqWG} has Hausdorff dimension at least $\min\{\gamma,1\}$. Hence $\dim_{\rm H}W_{\mathcal{A},\mathcal{B}}(\psi)\geq \min\{\gamma,1\}$. This completes the proof of Theorem \ref{thm1Haus}.
\end{proof}

\section{Proof of Theorem \ref{ThmSabPsi}}\label{Sec4}

In  this section, we give the proof of Theorem \ref{ThmSabPsi}. 

Given two sequences of positive integers $\mathcal{A}=\{a_n\}_{n\in\N}$ and $\mathcal{B}=\{b_n\}_{n\in\N}$ and two positive functions $\psi, \phi:\N\to (0,1)$, 
let $$S_{\mathcal{A},\mathcal{B}}(\psi,\phi)=\{x\in[0,1]: \|a_nx\|<\psi(n)\text{ and }\|b_nx\|<\phi(n)\text{ for i.m. }n\in\N\}.$$
Let $s\in (0,1]$.
Then from Lemma \ref{corcovering} we see that if 
\begin{equation}\label{eqSimuCon}
\sum_{n=1}^{\infty}\left[g_n+a_nb_n\max\left\{\frac{\psi(n)}{a_n}, \frac{\phi(n)}{b_n}\right\}\right]\min\left\{\frac{\psi(n)}{a_n}, \frac{\phi(n)}{b_n}\right\}^s<\infty,	
\end{equation}
then $\mathcal{H}^s(S_{\mathcal{A},\mathcal{B}}(\psi,\phi))=0$, where $g_n={\rm gcd}(a_n,b_n)$.
In the special case when $\psi=\phi$, set 
$$S_{\mathcal{A},\mathcal{B}}(\psi)=\{x\in[0,1]: \max\{\|a_nx\|, \|b_nx\|\}<\psi(n)\text{ for i.m. }n\in\N\}.$$
If in addition that  $a_n\leq b_n$ for all $n$, then the above convergence result is simplified to  the following:
\begin{equation}\label{eqanlesbnspsi}
\sum_{n=1}^{\infty}\left(g_n+b_n\psi(n)\right)\left(\frac{\psi(n)}{b_n}\right)^s<\infty \Longrightarrow \mathcal{H}^s(S_{\mathcal{A},\mathcal{B}}(\psi))=0.
\end{equation} 
Consequently, we have 
\begin{equation}\label{eqKappa}
\dim_{\rm H}S_{\mathcal{A}, \mathcal{B}}(\psi)\leq \kappa,
\end{equation}
where
\begin{equation}\label{eqDefKappa}
\kappa=\inf\left\{s>0: \sum_{n=1}^{\infty}(g_n+b_n\psi(n))\left(\frac{\psi(n)}{b_n}\right)^s<\infty\right\}.	
\end{equation}

Now we present the proof of Theorem \ref{ThmSabPsi}.

\begin{proof}[Proof of Theorem \ref{ThmSabPsi}]
Let $g={\rm gcd}(a,b)$.
Notice that since $\tau>1$,  for each $s\in (0,1)$ the series
\begin{equation}\label{eqSumSeries}
\sum_{n=1}^{\infty}(g^n+b^n\psi(n))\left(\frac{\psi(n)}{b^{n}}\right)^s=\sum_{n=1}^{\infty}g^n\left(\frac{\psi(n)}{b^n}\right)^s+\sum_{n=1}^{\infty}b^{n(1-s)+(1+s)\log_b\psi(n)}
\end{equation}
 has the same convergence and divergence property with $\sum_{n=1}^{\infty}g^n\left(\frac{\psi(n)}{b^n}\right)^s$, since the second series in the right-hand side of \eqref{eqSumSeries} converges. Hence the series in \eqref{eqSumSeries} converges  for any $s>\frac{\log_b{\rm gcd}(a,b)}{(1+\tau)}$, and so by \eqref{eqKappa}-\eqref{eqDefKappa}, we have 
\begin{equation}\label{eqDimSabLeq}
\dim_{\rm H}S_{a,b}(\tau)\leq \frac{\log_b{\rm gcd}(a,b)}{(1+\tau)}.
\end{equation}
On the other hand,  Lemma \ref{LemSubset} implies that $S_{a,b}(\tau)$ contains the set 
\[\left\{x\in [0,1]: \|g^nx\|<g^n\left(\frac{\psi(n)}{b^n}\right)\text{ for i.m. }n\in\N\right\},\]
which is known to have  Hausdorff dimension $\frac{\log_b{\rm gcd}(a,b)}{(1+\tau)}$ (cf. \cite{HilVelani95T}; see also \cite{SLMWBW13S}). Hence the reverse inequality in \eqref{eqDimSabLeq} holds. This completes the proof the theorem.	
\end{proof}

\section{Final remarks}\label{SecRemark}

In this section, we give some remarks concerning the sharpness and extensions of our results.

\subsection{Divergence results for Hausdorff measures} In Theorem \ref{thm1HMeas}, for each of  the sets $W^{\times}_{\mathcal{A},\mathcal{B}}(\psi)$ and $W_{\mathcal{A},\mathcal{B}}(\psi)$, we give a condition in terms of convergence of a certain series  so that the Hausdorff measure equals zero. One may wonder if the condition actually provides a dichotomy for the Hausdorff measure to be zero or infinity; i.e., whether the set has infinite Hausdorff measure if the series diverges. We are unable to prove this in  the full generality of Theorem \ref{thm1HMeas}. However, we point out that for some classes of $\mathcal{A}$ and $\mathcal{B}$ the answer is affirmative.

An infinite subset  of positive integers $\mathcal{A}=\{a_n\}_{n\in\N}$ is said to be {\em lacunary} if there there exists a constant $K>1$  such that for all $n\geq 1$,
\[\frac{a_{n+1}}{a_n}\geq K.\]
Given such an $\mathcal{A}$ and a nonnegative function $\psi:\N\to\R_{\geq0}$ which is not necessarily monotonic, it is known that (cf. \cite[Theorem 7.3]{Harman98}) the Lebesgue measure of the set 
\[W_{\mathcal{A}}(\psi)=\{x\in [0,1]: \|a_nx\|<\psi(n)\text{ for i.m. }n\in\N\}\]
satisfies the following zero-one dichotomy:
\begin{equation}\label{eq01Leb}
\mathcal{L}(W_{\mathcal{A}}(\psi))=\begin{cases}0, & \text{ if }\sum_{n=1}^{\infty}\psi(n)<\infty,\\
1, & \text{ if }\sum_{n=1}^{\infty}\psi(n)=\infty.
\end{cases}
\end{equation}
Far-reaching generalizations of this result were recently obtained in \cite{PVZZ22I}. By \eqref{eq01Leb} and a standard application of the mass transference principle established in \cite{VBSV06}, we see that for $s\in (0,1)$, the $s$-dimensional Hausdorff measure of $W_{\mathcal{A}}(\psi)$ satisfies a zero-infinity dichotomy as follows:
\begin{equation}\label{eq0infHaus}
\mathcal{H}^s(W_{\mathcal{A}}(\psi))=\begin{cases}0, & \text{ if }\sum_{n=1}^{\infty}a_n\left(\frac{\psi(n)}{a_n}\right)^s<\infty,\\
\infty, & \text{ if }\sum_{n=1}^{\infty}a_n\left(\frac{\psi(n)}{a_n}\right)^s=\infty.
\end{cases}
\end{equation}
Based on this fact and our result Theorem \ref{thm1HMeas}, we have the following.

\begin{thm}\label{ThmLacunary}
Let $W^{\times}_{\mathcal{A},\mathcal{B}}(\psi)$ and $W_{\mathcal{A},\mathcal{B}}(\psi)$ be as in Theorem \ref{thm1HMeas}. Let $s\in (0,1)$. If $\mathcal{A}$ and $\mathcal{B}$ are both lacunary, then we have 
\begin{equation}\label{eqMtimesHaus}
\mathcal{H}^{1+s}(W^{\times}_{\mathcal{A},\mathcal{B}}(\psi))=\begin{cases}0, & \text{ if }\sum_{n=1}^{\infty}\left[a_n\left(\frac{\psi(n)}{a_n}\right)^s+b_n\left(\frac{\psi(n)}{b_n}\right)^s\right]<\infty,\\
\infty, & \text{ if }\sum_{n=1}^{\infty}\left[a_n\left(\frac{\psi(n)}{a_n}\right)^s+b_n\left(\frac{\psi(n)}{b_n}\right)^s\right]=\infty.
\end{cases}
\end{equation}
If in addition that $\mathcal{G}=\{g_n\}_{n\in\N}$ is also lacunary, where $g_n={\rm gcd}(a_n,b_n)$, then
\begin{equation}\label{eqdivWAB}
\mathcal{H}^s(W_{\mathcal{A},\mathcal{B}}(\psi))=\begin{cases}0, & \text{ if }\sum_{n=1}^{\infty}\left[a_n\left(\frac{\psi(n)}{a_n}\right)^s+g_n\left(\frac{\psi(n)}{a_nb_n}\right)^{\frac{s}{2}}+b_n\left(\frac{\psi(n)}{b_n}\right)^s\right]<\infty,	\\
\infty, & \text{ if }\sum_{n=1}^{\infty}\left[a_n\left(\frac{\psi(n)}{a_n}\right)^s+g_n\left(\frac{\psi(n)}{a_nb_n}\right)^{\frac{s}{2}}+b_n\left(\frac{\psi(n)}{b_n}\right)^s\right]=\infty.
\end{cases}
\end{equation} 
\end{thm}

\begin{proof}
The convergence parts of \eqref{eqMtimesHaus} and \eqref{eqdivWAB} follow from Theorem \ref{thm1HMeas}.  If  the series in \eqref{eqMtimesHaus} diverges, then either $\sum_{n=1}^{\infty}a_n\left(\frac{\psi(n)}{a_n}\right)^s=\infty$ or $\sum_{n=1}^{\infty}b_n\left(\frac{\psi(n)}{b_n}\right)^s=\infty$. Thus by \eqref{eq0infHaus}, we have either $\mathcal{H}^s(W_{\mathcal{A}}(\psi))=\infty$ or $\mathcal{H}^s(W_{\mathcal{B}}(\psi))=\infty$. Notice that 
\[W_{\mathcal{A},\mathcal{B}}^{\times}(\psi)\supseteq W_{\mathcal{A}}(\psi)\times [0,1] \quad \text{ and } \quad W_{\mathcal{A},\mathcal{B}}^{\times}(\psi)\supseteq  [0,1]\times  W_{\mathcal{B}}(\psi).\]
It then follows from \cite[Theorem 5.8]{Falconer86T} that $\mathcal{H}^{1+s}(W_{\mathcal{A},\mathcal{B}}^{\times}(\psi))=\infty$.

Next, suppose  the series in \eqref{eqdivWAB} diverges. Then at least one of the following three series diverges: 
\begin{equation*}\label{eq3series}
\sum_{n=1}^{\infty}a_n\left(\frac{\psi(n)}{a_n}\right)^s, \quad \sum_{n=1}^{\infty}g_n\left(\frac{\psi(n)}{a_nb_n}\right)^{\frac{s}{2}}, \quad \sum_{n=1}^{\infty}b_n\left(\frac{\psi(n)}{b_n}\right)^s.
\end{equation*}
We have seen in the proof of Theorem \ref{thm1Haus} (cf. \eqref{eqWG}) that  
\[W_{\mathcal{A},\mathcal{B}}(\psi)\supseteq W_{\mathcal{A}}(\psi), \quad  W_{\mathcal{A},\mathcal{B}}(\psi)\supseteq   W_{\mathcal{B}}(\psi), \quad \text{and}\quad W_{\mathcal{A},\mathcal{B}}(\psi)\supseteq W_{\mathcal{G}}\left(g_n\left(\frac{\psi(n)}{a_nb_n}\right)^{\frac{1}{2}}\right).\]
Hence by \eqref{eq0infHaus} we have $\mathcal{H}^s(W_{\mathcal{A},\mathcal{B}}(\psi))=\infty$, completing the proof of the theorem.
\end{proof}

Theorem \ref{thm1Haus} can be also applied to give zero-infinity dichotomy for Hausdorff measures of $W_{\mathcal{A},\mathcal{B}}^{\times}(\psi)$ and  $W_{\mathcal{A},\mathcal{B}}(\psi)$ for some   $\mathcal{A}, \mathcal{B}$ which are not necessarily lacunary. To present such an example, we make use a recent result of \cite{PVZZ22I}. Given a set  $\mathcal{S}=\{p_1,\ldots, p_k\}$ of $k$ distinct prime numbers, let 
\begin{equation}\label{eqDefQS}
Q_{\mathcal{S}}=\left\{\prod_{i=1}^{k}p_i^{t_i}: t_1,\ldots, t_k\in \Z_{\geq0}\right\}
\end{equation}
be the set of positive integers with prime divisors restricted to $\mathcal{S}$. Let $\mathcal{A}=\{a_n\}_{n\in\N}\subseteq Q_{\mathcal{S}}$ be an increasing sequence of natural numbers. Then according to \cite[Corollary 2]{PVZZ22I},  the zero-one dichotomy \eqref{eq01Leb} holds, and again by the mass transference principle we have \eqref{eq0infHaus}. Based on this and an argument similar to the proof of Theorem \ref{ThmLacunary}, we have the following result.

\begin{thm}\label{ThmNoNLacunary}
Let $\mathcal{S}_1, \mathcal{S}_2$ be two finite sets of prime numbers, $\mathcal{A}=\{a_n\}_{n\in \N}\subseteq Q_{\mathcal{S}_1}, \mathcal{B}=\{b_n\}_{n\in\N}\subseteq Q_{\mathcal{S}_2}$ be two  increasing sequences of natural numbers. 
Let $W^{\times}_{\mathcal{A},\mathcal{B}}(\psi)$ and $W_{\mathcal{A},\mathcal{B}}(\psi)$ be as in Theorem \ref{thm1HMeas} and  $s\in (0,1)$. Then \eqref{eqMtimesHaus} holds.  Moreover, let $\mathcal{G}=\{g_n\}_{n\in \N}$ with $g_n={\rm gcd}(a_n,b_n)$. Then \eqref{eqdivWAB} holds in each of the following two cases: {\rm (i)} $\mathcal{G}$ is bounded; {\rm (ii)} $\mathcal{G}$ is an increasing sequence.
\end{thm}

\begin{proof}
According to the paragraph preceding the theorem,  \eqref{eq0infHaus} holds for $\mathcal{A},\mathcal{B}$ and $\mathcal{G}$ when $\mathcal{G}$ is an increasing sequence. Then  by a similar argument as in the proof of Theorem \ref{ThmLacunary}, we see that \eqref{eqMtimesHaus} holds,  and \eqref{eqdivWAB} holds in the case (ii).  

To prove \eqref{eqdivWAB} in the case (i), assume $\mathcal{G}$ is bounded. By a result of Marstrand \cite[p.545]{Marstrand70O} (see also \cite{FQQ24T}) on the distribution of $Q_{\mathcal{S}_1}$, we have 
\[a_n\gg e^{c_k\sqrt[k]{n}},\]
where $k$ is the number of elements in $\mathcal{S}_1$, $c_k$ is a positive constant depending only on $k$, and the implicit constant in ``$\gg$'' is independent of $n$. It then follows that for every $s\in (0,1)$,  the series $\sum_{n=1}^{\infty}g_n\left(\frac{\psi(n)}{a_nb_n}\right)^{\frac{s}{2}}$ convergences. Hence the series in \eqref{eqdivWAB} has the same convergence/divergence property with the series in \eqref{eqMtimesHaus}. So a similar reasoning as above yields  \eqref{eqdivWAB}.
\end{proof}

\subsection{An example} Notice that in Theorem \ref{thm1HMeas} (and thus in Theorem \ref{thm1Haus}), the series involved in our results for  $W_{\mathcal{A},\mathcal{B}}^{\times}(\psi)$ and $W_{\mathcal{A},\mathcal{B}}(\psi)$ are of different form: there is an extra term
\begin{equation*}\label{eqgcdterm}
{\rm gcd}(a_n,b_n)\left(\frac{\psi(n)}{a_nb_n}\right)^{\frac{s}{2}}	
\end{equation*} 
in the series for $W_{\mathcal{A},\mathcal{B}}(\psi)$. This is not surprising, and one can easily construct examples to show that the term does play a role and hence cannot be omitted in general. For instance, let $s=\frac{1}{2}$, $a_n=4^n$, $b_n=8^n$ and $\psi(n)=\frac{1}{n^42^{3n}}$ for $n\in\N$. Then it is easily checked that  
\[\sum_{n=1}^{\infty}\left[a_n\left(\frac{\psi(n)}{a_n}\right)^s+b_n\left(\frac{\psi(n)}{b_n}\right)^s\right]<\infty, \quad \sum_{n=1}^{\infty}{\rm gcd}(a_n,b_n)\left(\frac{\psi(n)}{a_nb_n}\right)^{\frac{s}{2}}=\infty.\]
We thus have by Theorem \ref{ThmLacunary} that $\mathcal{H}^{s}(W_{\mathcal{A},\mathcal{B}}(\psi))=\infty$.

\subsection{Lebesgue measure}
In this paper, we mainly concern about the Hausdorff measures and dimensions of $W^{\times}_{\mathcal{A},\mathcal{B}}(\psi)$ and $W_{\mathcal{A},\mathcal{B}}(\psi)$. As for the Lebesgue measures of these sets, we only have some partial results. First notice that Lemma \ref{lemLebesgue} leads to the following convergence result of the Lebesgue measure of  $W_{\mathcal{A},\mathcal{B}}(\psi)$.

\begin{thm}\label{thmLebes}
Let $W_{\mathcal{A},\mathcal{B}}(\psi)$ be as in Theorem \ref{thm1HMeas}. 
If 
\begin{equation}\label{dimsstar58}
\sum_{n=1}^{\infty}\left[{\rm gcd}(a_n,b_n)\left(\frac{\psi(n)}{a_nb_n}\right)^{\frac{1}{2}}+\psi(n)\log\left(\frac{1}{\psi(n)}\right)\right]<\infty,
\end{equation}
then we have $\mathcal{L}(W_{\mathcal{A},\mathcal{B}}(\psi))=0$.
\end{thm}

\begin{proof}
For $n\in \N$, let $F_n=\{x\in [0,1]: \|a_nx\|\|b_nx\|<\psi(n)\}$. Then  $W_{\mathcal{A},\mathcal{B}}(\psi)=\limsup_{n\to\infty}F_n$. 
By Lemma \ref{lemLebesgue} (in which we take $\delta^2=\psi(n)$), there is an absolute constant $C_3$ such that 
\begin{equation*}
\mathcal{L}(F_n)\leq C_3\left[{\rm gcd}(a_n,b_n)\left(\frac{\psi(n)}{a_nb_n}\right)^{\frac{1}{2}}+\psi(n)\log\left(\frac{1}{\psi(n)}\right)\right].	
\end{equation*}
The theorem then follows by the Borel-Cantelli lemma.
\end{proof}
 
Similar to the case for Hausdorff measure, we have divergence results for Lebesgue measure only in some restrictive circumstances. For instance, based on the above result and some known results on one dimensional approximation, we have the following.

\begin{pro}
Suppose that $\varliminf_{n\to\infty}\frac{\log \psi(n)^{-1}}{\log n}>1$. Then in each of the following two cases:
\begin{itemize}
\item[(i)] $\{g_n\}_{n\in\N}$ is lacunary,

\item[(ii)] $\{g_n\}_{n\in\N}\subseteq Q_{\mathcal{S}}$ is an increasing sequence with $Q_{\mathcal{S}}$ being defined in \eqref{eqDefQS}, 
\end{itemize}
we have 
\[\mathcal{L}(W_{\mathcal{A},\mathcal{B}}(\psi))=\begin{cases}0, & \text{ if }\sum_{n=1}^{\infty}g_n\left(\frac{\psi(n)}{a_nb_n}\right)^{\frac{1}{2}}<\infty,\\
1, & \text{ if }\sum_{n=1}^{\infty}g_n\left(\frac{\psi(n)}{a_nb_n}\right)^{\frac{1}{2}}=\infty.
\end{cases}\]
\end{pro}

\begin{proof}
Notice that the assumption that $\varliminf_{n\to\infty}\frac{\log \psi(n)^{-1}}{\log n}>1$ guarantees that the series $\sum_{n=1}^{\infty}\psi(n)\log\left(\frac{1}{\psi(n)}\right)$ converges. Hence the series in \eqref{dimsstar58} has the same convergence/divergence property with $\sum_{n=1}^{\infty}g_n\left(\frac{\psi(n)}{a_nb_n}\right)^{\frac{1}{2}}$. Therefore the convergence part of the proposition follows from \eqref{dimsstar58}. As for the divergence part, observe that by Lemma \ref{LemSubset}, $W_{\mathcal{A},\mathcal{B}}(\psi)$ contains the set 
\[\left\{x\in [0,1]: \|g_nx\|<g_n\left(\frac{\psi(n)}{a_nb_n}\right)^{\frac{1}{2}}\text{ for i.m. }n\in\N\right\}.\]
Then the divergence part of the proposition  follows from \cite[Theorem 7.3]{Harman98}, \cite[Corollary 2]{PVZZ22I}, and \eqref{eq01Leb} (in which we let $\mathcal{A}$ be $\{g_n\}_{n\in\N}$ and $\psi(n)$ be $g_n\left(\frac{\psi(n)}{a_nb_n}\right)^{\frac{1}{2}}$).
\end{proof}

{\noindent \bf  Acknowledgements}. The authors would like to thank Lingmin Liao, Baowei Wang and Bo Wang for helpful comments and suggestions. B. Li was supported by NSFC12271176 and Guangdong Natural Science Foundation 2024A1515010946. R. F. Li was supported by NSFC12401006 and Guangdong Basic and Applied Basic Research Foundation 2023A1515110272. Y. F. Wu (corresponding author) was supported by NSFC12301110.

\end{document}